\documentclass[11pt]{amsart}

\usepackage[a4paper,hmargin=3.5cm,vmargin=4cm]{geometry}
\usepackage{amsfonts,amssymb,amscd,amstext}
\usepackage{graphicx}
\usepackage[dvips]{epsfig}

\usepackage{fancyhdr}
\input xy
\xyoption{all}

\usepackage{times}

\usepackage{enumerate}
\usepackage{titlesec}
\usepackage{mathrsfs}

\titleformat{\section}
{\filcenter\bfseries\large} {\thesection{.}}{0.2cm}{}
\titleformat{\subsection}[runin]
{\bfseries} {\thesubsection{.}}{0.15cm}{}[.]
\titleformat{\subsubsection}[runin]
{\em}{\thesubsubsection{.}}{0.15cm}{}[.]

\usepackage[up,bf]{caption}

\newtheorem{theorem}{Theorem}[section]

\newtheorem{lemma}[theorem]{Lemma}
\newtheorem{corollary}[theorem]{Corollary}

\theoremstyle{definition}

\numberwithin{equation}{section}
\numberwithin{figure}{section}

\def\ccal{\mathcal{C}}

\def\Fcal{\mathcal{F}}

\def\c{\mathbb{C}}

\def\b{\mathbb{B}}
\def\r{\mathbb{R}}

\def\igot{\mathfrak{i}}

\def\igot{\mathfrak{i}}

\renewcommand\imath{\igot}

\def\diam{\mathrm{diam}\,}
\def\dist{\mathrm{dist}}
\def\grad{\mathrm{grad}}

\def\length{\mathrm{length}}

\def\cd{\overline{D}}

\usepackage{color}

\begin{document}

\fancyhead[LO]{Complete proper holomorphic embeddings into balls} 

\fancyhead[RE]{ B.\ Drinovec Drnov\v sek}

\fancyhead[RO,LE]{\thepage}

\thispagestyle{empty}

\vspace*{1cm}
\begin{center}
{\bf\LARGE Complete proper holomorphic embeddings of strictly pseudoconvex domains into balls}

\vspace*{0.5cm}

{\large\bf Barbara Drinovec Drnov\v sek}
\end{center}


\vspace*{1cm}

\begin{quote}
{\small
\noindent {\bf Abstract}\hspace*{0.1cm}
We construct a complete proper holomorphic embedding from any strictly pseudoconvex domain with 
$\ccal^2$-boundary in $\c^n$ into
the unit ball of $\c^N$, for $N$ large enough,
thereby answering a question of Alarc\'on and Forstneri\v c \cite{AF}. 

\vspace*{0.1cm}

\noindent{\bf Keywords}\hspace*{0.1cm} complete map, proper holomorphic map, peak function.

\vspace*{0.1cm}


\noindent{\bf MSC (2010):}\hspace*{0.1cm} 32H35, 32H02, 32T40}
\end{quote}


\section{Introduction} 
\label{sec:intro}

The question of existence of complete bounded submanifolds in $\c^n$ was raised by Yang in 1977 \cite{Yo1,Yo2}, and even before, in 1965, Calabi 
conjectured the nonexistence of complete minimal surfaces in $\r^3$
with bounded projection into a straight line \cite{Ca},
which turned out to be false \cite{JX}.
These inspired many results in complex analysis and minimal surface theory. For a survey of the results and references,
see the introduction in \cite{AF} and the survey \cite{AF1}.
Most of the known results regarding Yang's question hold for complex curves,
including the first positive answer by Jones \cite{Jo},
whereas for higher dimensional submanifolds not much was known until recently:
Globevnik \cite{Gl} proved that for any $n$, $m$, $1\le n<m$,  there is a complete closed $n$-dimensional complex submanifold in 
the unit ball of $\c^m$, and therefore he completely answered Yang's question.
In his construction there is no control on the topology of the submanifolds.

Alarc\'on and Forstneri\v c \cite{AF} constructed a complete proper holomorphic immersion from any bordered Riemann surface 
into the unit ball in $\c^2$, and a complete proper holomorphic embedding into the unit ball in $\c^m$, $m\ge 3$.
They used the method of exposing boundary points of a complex curve in $\c^2$
\cite{FW} together with  the approximate solution to a Riemann-Hilbert boundary value problem. None of these is available in higher dimensions.
They asked if there is a complete proper holomorphic immersion/embedding from the unit ball in $\c^n$ into the unit ball
of a higher dimensional Euclidean space. The aim of this note is to give an affirmative answer to their question.

Let $\b_m$ denote the open unit ball in $\c^m$. An embedding $f\colon D\to \c^m$ from an open subset $D\subset \c^n$ is {\em complete}
if the induced Riemannian metric $f^*ds^2$ on $D$ obtained by pulling back the Euclidean metric $ds^2$ on $\c^m$ is a complete
metric on $D$.
The main result of this note is the following theorem:

\begin{theorem}
	\label{main_theorem}
	Let $D$ be a bounded strictly convex domain with $\ccal^2$-boundary in $\c^n$. There exists a positive integer $s$ with the 
	following property. For any positive integer $p$ and for any continuous map $h\colon \cd\to \b_p$, which is an injective holomorphic immersion in $D$,
	there exists a holomorphic map $f\colon D\to\c^{2s}$, such that the map $(f,h)\colon D\to \b_{2s+p}$ is
a	complete proper holomorphic embedding.
\end{theorem}



The main ingredient in the proof are holomorphic peak functions, the idea which goes back to Hakim and Sibony \cite{HS} and L\o w \cite{Lo1},
and the construction of inner functions on the unit ball. More precisely, we refine the construction of Forstneri\v c \cite{FF} of
a proper holomorphic  map from a strictly convex domain with $\ccal^2$-boundary in $\c^n$
into a unit ball of some Euclidean space; see also \cite{Lo3} where the author 
obtained in addition to the above, a proper holomorphic map into a higher
dimensional unit ball, which extends continuously to the boundary.
Note that recently Globevnik \cite{Gl2} proved that there are no complete proper holomorphic maps from the
open unit disc in $\c$ to the open unit bidisc in $\c^2$ which extend continuously to the boundary.

By Fornaess' embedding theorem \cite{Fo} any bounded strictly pseudoconvex domain with $\ccal^2$-boundary embeds
properly holomorphically into a strictly convex domain in Euclidean space. Since the composition of a proper and a complete
proper holomorphic embedding is a complete proper holomorphic embedding we have the following corollary.

\begin{corollary}
	Let $D$ be a bounded strictly pseudoconvex domain with $\ccal^2$-boundary in $\c^n$.	
	For $N$ large enough there exists a complete proper holomorphic embedding 
	$F\colon D\to \b_N$.
\end{corollary}

Note that one could also extend the construction in the proof of Theorem \ref{main_theorem} 
to obtain the same result where the domain $D$ is
strictly pseudoconvex using the arguments of L\o w \cite{Lo2}.
More precisely, we could use different holomorphic peak functions 
with estimates similar to Lemma 2.1 below, see \cite[Lemma 2.7]{Lo2}.



\section{Proof of Theorem \ref{main_theorem}} 
\label{sec:proof}

Throughout this section, $D$ is a bounded strictly convex domain with $\ccal^2$-boundary in $\c^n$.
Let $S$ denote its boundary and $\nu(w)$ the outward unit normal to $S$ at the point $w\in S$.
For $a\in\c^n$ and $r>0$ let $\b(a,r)$ denote the open ball of radius $r$ centered at $a$ in $\c^n$.
We denote by $\langle\cdot,\cdot\rangle$, $\|\cdot\|$, and $\dist(\cdot,\cdot)$ the  Hermitian inner product, norm, and distance in $\c^n$.


The following lemma is a slight generalization of \cite[Lemma 5.1]{FF}.

\begin{lemma}
\label{Lemma5.1FF}
	There are constants $\alpha_1$, $\alpha_2$, $r_1>0$ such that the following hold:
	\begin{equation}
	\begin{split}
	\label{eq:Lemma5.1FF}
	\Re \langle w-z,\nu(w)\rangle&\ge \alpha_1 \| z-w\|^2 \text{ for all } w\in S, z\in \cd \text{ such that } \dist(z,bD)<r_1,\\
  \Re \langle w-z,\nu(w)\rangle&\le \alpha_2 \| z-w\|^2 \text{ for all }z,w\in S.
	\end{split}
	\end{equation}
\end{lemma}

\begin{proof}
The existence of $\alpha_2>0$ satisfying the second estimate was already a part of \cite[Lemma 5.1]{FF}.
Let $\rho$ denote a $\ccal^2$-defining function of $D$ such that $\{z\colon\rho(z)<0\}=D$ and
$\grad \rho(z)$ does not vanish for any $z\in bD$. Then there exists $\gamma_1>0$ such that
$\grad \rho(z)$ does not vanish for any $z$, $\rho(z)\in  [-\gamma_1,0]$,  and
the proof of \cite[Lemma 5.1]{FF} provides a constant $\alpha_1>0$ such that
\begin{equation}
\label{eq:lemma21}
 \Re \langle w-z,\nu(w)\rangle \ge \alpha_1 \| z-w\|^2 \text{ for all } z,w\in \cd\text{ such that } \rho(z)=\rho(w)\in [-\gamma_1,0].
\end{equation}
We may assume that $\alpha_1>0$ is so small that $1-\alpha_1(1+2\diam D)>0$, where $\diam D$ denotes the diameter of $D$.
Since the boundary $bD$ is of class $\ccal^2$ we can choose $r_1$, $0<r_1<\frac 12$, so small that 
\begin{equation}
\label{eq:lemma21a}
	\begin{split}
	&\{z\in \cd\colon \dist (z,bD)<r_1\}\subset \{w-r\nu(w)\colon w\in S,\ r\in[0,2r_1]\}\cap\rho^{-1}([-\gamma_1,0]),\\
	&\|\nu(w-r\nu(w))-\nu(w)\|\le \frac{1-\alpha_1(1+2\diam D)}{\diam D} r \text{ for all } w\in S, r\in [0,2r_1].
	\end{split}
\end{equation}
By the choice of $r_1$, for any  $w\in S$ and  $z\in \cd$ such that $\dist(z,bD)<r_1$ there is
$r\in[0,2r_1]$ such that $\rho(w-r\nu(w))=\rho(z)\in[-\gamma_1,0]$. Letting $w'=w-r\nu(w)$ we get
\begin{equation}
	\label{lema21b}
	\begin{split}
		\Re \langle w-z,\nu(w)\rangle&=\Re \langle w'-z,\nu(w)\rangle+r\\
		&=\Re \langle w'-z,\nu(w')\rangle+\Re \langle w'-z,\nu(w)-\nu(w')\rangle + r\\
		&\stackrel{\eqref{eq:lemma21},\eqref{eq:lemma21a}}{\ge}\alpha_1 \|w'-z\|^2  - (1-\alpha_1(1+2\diam D)) r + r.
	\end{split}
\end{equation}
On the other hand, we have
\begin{equation*}
	\begin{split}
		\|w-z\|^2&=\|w-w'+w'-z\|^2\\
		&\le r^2+\|w'-z\|^2+2r \|w'-z\|\\ 
		&\le \|w'-z\|^2 + (1+2\diam D)r.
	\end{split}
\end{equation*}
By \eqref{lema21b} we obtain 
$\Re \langle w-z,\nu(w)\rangle\ge \alpha_1\|w-z\|^2$,
which completes the proof.
\end{proof}

For the convenience of the reader we recall the next covering lemma from \cite{FF}, see also \cite{Lo4}:

\begin{lemma}\cite[Lemma 5.2]{FF}
\label{Lemma5.2FF}
	For every $\lambda>1$ there exists an integer $s>0$ with the following property:
	For each $r>0$ there are $s$ families of balls $\Fcal_1,\ldots,\Fcal_s$,
	\[
	\Fcal_i=\{\b(z_{i,j},\lambda r):1\le j\le N_i\},\]
	 with centers $z_{i,j}\in S$,
	such that the balls in each family are pairwise disjoint, and 
	\begin{equation}
		\label{eq:Lemma5.2FF}
		S\subset \bigcup_{i=1}^s\bigcup_{j=1}^{N_i}\b(z_{i,j}, r). 
	\end{equation}
\end{lemma}

Let $\alpha_1$ and $\alpha_2$ be as in Lemma \ref{Lemma5.1FF} and
let 
\begin{equation}
 \label{eq:lambda}
 \lambda=4\sqrt{\frac{\alpha_2}{\alpha_1}}.
\end{equation}
Note that our choice of the constant $\lambda$  is different from the one in \cite[(5.7)]{FF},
because we need more precise estimates in the next lemma.

For the chosen $\lambda$ we get a positive integer $s$ satisfying the properties
in Lemma \ref{Lemma5.2FF}. Therefore, for any $r>0$ we have $s$ families of balls  $\Fcal_1,\ldots,\Fcal_s$,
$\Fcal_i=\{\b(z_{i,j},\lambda r):1\le j\le N_i\}$, $z_{i,j}\in S$, such that the balls in each $\Fcal_i$ are pairwise disjoint
and balls with the same centers and radii $r$ cover $S$ \eqref{eq:Lemma5.2FF}.

For each $1\le i\le s$ and  $1\le j\le N_i$ we define $z_{i+s,j}=z_{i,j}$ and $\Fcal_{i+s}=\Fcal_i$.
Further, for $m>0$, $1\le i\le 2s$ and $1\le j\le N_i$ we define
\begin{equation}
	\label{eq:defphi}
 	\phi_{i,j}(z)=e^{-m\langle z_{i,j}-z,\nu(z_{i,j})\rangle},\ z\in\cd.
\end{equation}
By \eqref{eq:Lemma5.1FF} we get the following estimates
\begin{equation}
\begin{split}
\label{eq:esimate5.1}
|\phi_{i,j}(z)|&\le e^{-\alpha_1m\|z-z_{i,j}\|^2}\quad \text{ for all } w\in S, z\in \cd \text{ such that } \dist(z,bD)<r_1,\\
|\phi_{i,j}(z)|&\ge e^{-\alpha_2m\|z-z_{i,j}\|^2}\quad \text{ for each } z\in S.
\end{split}
\end{equation}
For given $|\beta_{i,j}|\le 1$, let $g_i$ be the entire function
\begin{equation}
	\label{eq_5.11}
	g_i(z)=\sum_{j=1}^{N_i} \beta_{i,j} \phi_{i,j}(z),\qquad \ z\in \cd.
\end{equation}

The next lemma is similar to \cite[Lemma 5.3]{FF}, with the following differences:
The estimate in (b) holds on $\cd\cap \b(z_{i,j},\lambda r)$ whereas in 
\cite{FF} it holds on $S\cap \b(z_{i,j},\lambda r)$.
The growth of $\eta$ in (c) is different since we chose different $\lambda$, and 
the property (d) is added since it will be needed in the proof of  Theorem \ref{main_theorem}.

\begin{lemma}
\label{Lemma5.3FF} Let $r_1$, $\lambda$, $s$, $\Fcal_i$, $g_i$, $\beta_{i,j}$, and $\phi_{i,j}$ be as above.
	For each  sufficiently small $\eta>0$ there are $m,r>0$, $0<\lambda r<r_1$, such that for  each
	$i$, $1\le i\le 2s$,  the following hold for
	the family of balls $\Fcal_i$  and for the functions $g_i$:
	\begin{enumerate}[\rm(a)]
		\item If a point $z\in S$ lies in no ball in $\Fcal_i$, then $|g_i(z)|<\eta$.
		\item If $z\in \cd\cap \b(z_{i,j},\lambda r)$ for some $j$, then $|g_i(z)-\beta_{i,j}\phi_{i,j}(z)|<\eta$.
		\item If $z\in S\cap \b(z_{i,j}, r)$ for some $j$, then $|\phi_{i,j}(z)|\ge C\eta^{\frac1{16}}$, where the constant $C$ is
		independent of $r$, $m$ and $\eta$.
		\item If $z\in \cd\cap b\b(z_{i,j},\lambda r)$ for some $j$, then $|\phi_{i,j}(z)|<\eta^{\frac 23}$.
	\end{enumerate}
	Moreover, we can choose $r>0$ arbitrarily small and make $m>0$ as large as we want.   
\end{lemma}

\begin{proof}
Properties (a), (c) are proved the same way as in the proof of \cite[Lemma 5.3]{FF}.
We recall some parts of the proof, because we need the right choices of constants
in the proof of (d).

If $z\in S$ lies in no ball in $\Fcal_i$ then as in \cite{FF} we obtain 
 $|g_i(z)|<C_2e^{-\beta}$, where $\beta=16\alpha_2mr^2\ge \frac 4 3$
and the constant $C_2$ does not depend on $r$, $m$ or $\eta$.
Here $\beta$ is slightly different than in \cite{FF} since we chose a different $\lambda$.
Given $\eta>0$, we take  $m>0$ and $r>0$ such that $C_2e^{-\beta}=\eta$. 
If $\eta\le C_2e^{-\frac 4 3}$,
then $\beta\ge  \frac 4 3$ as needed.
Since 
\begin{equation}
\label{eq:mr2}
 mr^2=\frac 1 {16\alpha_2}\ln\frac{C_2}{\eta}  
\end{equation}
we can choose $r>0$ arbitrarily small and make $m$ as large as we want.
This proves (a). 

For the proof of (b) note that the  second estimate in \eqref{eq:esimate5.1} holds also on $\overline{\b}(z_{i,j},\lambda r)$, and then
the same proof as in \cite{FF} gives (b).

Take $z\in S\cap \b(z_{i,j}, r)$ and according to \eqref{eq:esimate5.1}
and \eqref{eq:mr2} we get
\[  |\phi_{i,j}(z)|\ge e^{-\alpha_2mr^2}= C_2^{-\frac1{16}}\eta^{\frac1{16}},     \]
which proves (c).

To prove (d), denote by $\pi\colon \cd\cap \b(z_{i,j},\lambda r)\to  S \cap \b(z_{i,j},\lambda r)$
the orthogonal projection to the boundary in the $\nu(z_{i,j})$ direction. 
By strict convexity, the map $\pi$ is well defined for 
any $r>0$ small enough, and 
for any $z\in\cd\cap \b(z_{i,j},\lambda r)$ there exists $s(z)\ge0$ such that
$z=\pi(z)-s(z)\nu(z_{i,j})$, i.e. $s(z)=\Re\langle \pi(z)-z,\nu(z_{i,j})\rangle$.
By \eqref{eq:defphi} we have
\begin{equation} 
\label{eq:lema5.3}
|\phi_{i,j}(z)|=e^{-m\Re\langle z_{i,j}-\pi(z),\nu(z_{i,j})\rangle}e^{-ms(z)}. 
\end{equation}
Both factors on the right are not bigger than 1.
We split $\cd\cap b\b(z_{i,j},\lambda r)$ into two parts 
in such a way that on each part one of the factors
is small enough to obtain the estimate (d). 
Fix any $\mu$, $\sqrt{\frac 23} \lambda<\mu<\lambda$.
Let $S_1=\cd \cap b\b(z_{i,j},\lambda r)\cap \pi^{-1}(S\cap \b(z_{i,j},\mu r))$ and 
$S_2=\cd \cap b\b(z_{i,j},\lambda r)\setminus S_1$.
For $z\in S_2$ we have $\|\pi(z)-z_{i,j}\|\ge  \mu r$, thus we get 
\[ |\phi_{i,j}(z)|\stackrel{\eqref{eq:lema5.3}}{\le} e^{-m\Re\langle z_{i,j}-\pi(z),\nu(z_{i,j})\rangle}
\stackrel{\eqref{eq:Lemma5.1FF}}{\le} e^{-\mu^2\alpha_1mr^2}\stackrel{\eqref{eq:mr2},\eqref{eq:lambda}}{=}
{\left(\frac{\eta}{C_2}\right)}^{\frac{\mu^2}{\lambda^2}}
< \eta^{\frac 23},\]
for each $\eta>0$ small enough.
For $z\in S_1$ we have $\|\pi(z)-z_{i,j}\|<  \mu r$, therefore   Pythagorean theorem and {\eqref{eq:Lemma5.1FF}}
imply
\[ s(z)=\Re\langle z_{i,j}-z,\nu(z_{i,j})\rangle-\Re\langle z_{i,j}-\pi(z),\nu(z_{i,j})\rangle
\ge \sqrt{\lambda^2-\mu^2} r -\alpha_2 \mu^2r^2.\]
For any given $\eta>0$,  we have
\[ |\phi_{i,j}(z)|\stackrel{\eqref{eq:lema5.3}}{\le} e^{-\sqrt{\lambda^2-\mu^2}  mr+\alpha_2 \mu^2mr^2}
\stackrel{\eqref{eq:mr2}}{=} e^{(-\frac 1r\sqrt{\lambda^2-\mu^2}+\alpha_2 \mu^2)(\frac 1{16\alpha_2})\ln{\frac{C_2}{\eta}}}
<\eta^{\frac 23},\]
where the last estimate holds for any  $r>0$ small enough.
This proves (d).
\end{proof}

The following lemma refines \cite[Lemma 6.1]{FF}.
The main addition is part (e) which guarantees that 
we increase the induced distance between a given point in $D$ and the 
boundary $S$ by a certain amount. 
Notice that the condition (iii) is slightly different from \cite[Lemma 6.1 (iii)]{FF};
it provides control of how much we gain in (e).
We shall denote the induced distance
by a map $F$ on $D$ by $\dist_F$.

\begin{lemma}
\label{Lemma6.1FF}
Let $D$, $S=bD$ and $h$ be as in the statement of Theorem \ref{main_theorem} and $s$ as above.
Then there is $\epsilon_0>0$ such that the following implication holds:
If we are given
\begin{enumerate}[\rm(i)]
	\item numbers $a$ and $\epsilon$, $0<\epsilon<\epsilon_0$, 
		such that $a-\epsilon^{\frac 12}>\frac 12$ and $a+\epsilon<1$, 
	\item a compact subset $K\subset D$,
	\item	a continuous map $f=(f_1,\ldots,f_{2s})\colon \cd\to\c^{2s}$, holomorphic in $D$, such that for the map $F=(f,h)$ we have  
		$\|F(z)\|<a-\epsilon^{\frac 12}$ for each $z\in S$,  
	\item a point $p\in D$ and a number $\sigma>0$ such that $\dist_{F}(p,S)>\sigma$, and
	\item a number $\delta>0$,
\end{enumerate}
then there exists an
entire mapping $G=(g_1,\ldots,g_{2s},0,\ldots,0)\colon \c^n\to\c^{2s+p}$ satisfying the following properties:
\begin{enumerate}[\rm(a)]
	\item $\|(F+G)(z)\|\le a+\epsilon$ for all $z\in S$,
	\item if $\|(F+G)(z)\|\le a-\epsilon^{\frac 1 7}$ for some  $z\in S$, then  
	$\|(F+G)(z)\|>\|F(z)\|+\epsilon^{\frac 2 7}$,
	\item $\|G(z)\|<\delta$ for all $z\in K$,
	\item $\|G(z)\|^2<1-\|F(z)\|$ for all $z\in S$,
	\item $\dist_{F+G}(p,S)>\sigma+E\epsilon^{\frac 5{16}}$,
where the constant $E$ depends only on $\epsilon_0$.
\end{enumerate}
\end{lemma}

Note that the fact that $h$ is an injective holomorphic immersion implies
that $F$ and $G$ are injective holomorphic immersions.

\begin{proof}
The proof of (a)-(d) follows the proof of \cite[Lemma 6.1]{FF}.
To obtain (e), we need a slightly different condition (iii)
and we need to choose different growth of $\epsilon$ in (b). 
The main idea of the proof is the same but we need to repeat the construction to make 
the necessary adjustments for the second part of the proof. 

Let $\eta=\frac \epsilon {120s}$. 
Let $r_1>0$ be the number provided by Lemma \ref{Lemma5.1FF} and $\lambda>0$ defined by \eqref{eq:lambda}.
By continuity of $F$ on $\cd$, there is $r_0$, $0<r_0<r_1$, such that 
for any $z,w\in\cd$ with $\|z-w\|<2\lambda r_0$ we have 
\begin{equation}
\label{eq:6.1FF}
	|f_i(z)-f_i(w)|<\eta,\quad 1\le i\le 2s,\quad \left| \|F(z)\|-\|F(w)\|\right|<\eta.
\end{equation}
Given $r$, $0<r<r_0$, to be chosen later we choose $s$ families of balls $\Fcal_1,\ldots,\Fcal_s$,
$\Fcal_i=\{\b(z_{i,j},\lambda r):1\le j\le N_i\}$, with centers $z_{i,j}\in S$,
such that the balls in each family are pairwise disjoint and the small balls also
cover $S$ (Lemma \ref{Lemma5.2FF}). Let $z_{i+s,j}=z_{i,j}$ and $\Fcal_{i+s}=\Fcal_i$, $1\le i\le s$.

We define the coefficients $\beta_{i,j}$ and $\beta_{i+s,j}$, $1\le i\le s$, $1\le j\le N_i$, \eqref{eq_5.11}, as follows:
\begin{equation}
\begin{split} \label{eq:6.26.3FF}
f_i(z_{i,j})\overline{\beta_{i,j}}+f_i(z_{i+s,j})\overline{\beta_{i+s,j}}&=0,\\
|\beta_{i,j}|^2+|\beta_{i+s,j}|^2&=\frac{a^2-\|F(z_{i,j})\|^2}{2s}.
\end{split}
\end{equation}

This implies that the vector $(\beta_{i,j},\beta_{i+s,j})$ is perpendicular to the vector
$(f_i(z_{i,j}),f_i(z_{i+s,j}))$ and $|\beta_{i,j}|<1$, $|\beta_{i+s,j}|<1$.
We shall prove that the entire map $G=(g_1,\ldots,g_{2s},0,\ldots,0)$, where $g_i$ are defined by
\eqref{eq_5.11} and satisfy Lemma \ref{Lemma5.2FF}, has the properties (a)-(e),
provided that the constant $m>0$ is chosen large enough and $r>0$ is chosen small enough. 

Part (a) is proved exactly as in the proof \cite[Lemma 6.1]{FF} and will not be repeated.
The proof of (b) is very similar but the choice of the constants is different,  so for the 
sake of the reader we repeat the relevant parts.
As in the proof of \cite[Lemma 6.1]{FF} we obtain:
For  $D_i(z)=|f_i(z)+g_i(z)|^2+|f_{i+s}(z)+g_{i+s}(z)|^2-|f_i(z)|^2-|f_{i+s}(z)|^2$ we have
\begin{equation}
\begin{split}
\label{eq:6.5FF}
D_i(z)&=\left(|\beta_{i,j}|^2+|\beta_{i+s,j}|^2\right)|\phi_{i,j}(z)|^2 + O(\epsilon),
\text{ if } z\in\b(z_{i,j},\lambda r)\text{ for some }j,\\
D_i(z)&=O(\epsilon),\quad z\text{ lies in no ball in }\Fcal_i,
\end{split}
\end{equation}
and furthermore
\begin{equation}
\label{eq:6.6FF}
\|F(z)+G(z)\|-\|F(z)\|\ge O(\epsilon).
\end{equation}
Suppose $\|(F+G)(z)\|\le a-\epsilon^{\frac 1 7}$ for some  $z\in S$. Choose a ball
$B(z_{i,j},r)$ containing $z$. Then we have 
\[ \|F(z_{i,j})\|\stackrel{\eqref{eq:6.1FF}}{\le} \|F(z)\| + \epsilon\stackrel{\eqref{eq:6.6FF}}{\le}
  a-\epsilon^{\frac 1 7} + O(\epsilon)<a-\frac 12\epsilon^{\frac 1 7} , \]
for any $\epsilon\in (0,\epsilon_0)$, if $\epsilon_0>0$ is chosen small enough. Therefore, since $a\ge \frac 12$ we get $a^2-\|F(z_{i,j})\|^2\ge \frac 14 \epsilon^{\frac 1 7}$, which 
implies by \eqref{eq:6.26.3FF} that $|\beta_{i,j}|^2+|\beta_{i+s,j}|^2\ge \frac 1 {8s}\epsilon^{\frac 1 7}$.
By Lemma \ref{Lemma5.3FF} (c), we obtain $|\phi_{i,j}(z)|^2\ge C^2  \eta^{\frac 1 8}$, which by \eqref{eq:6.5FF}
leads to
\[\|F(z)+G(z)\|^2-\|F(z)\|^2=\sum_{i=1}^s D_i(z)\ge \frac{C^2}{8s} \eta^{\frac 1 8}\epsilon^{\frac 1 7} + O(\epsilon)
\ge 2 \epsilon^{\frac 2 7},
\]
for any $\epsilon\in (0,\epsilon_0)$, if $\epsilon_0>0$ is chosen small enough.
Then we get 
\[  \|F(z)+G(z)\|-\|F(z)\|=\frac{\|F(z)+G(z)\|^2-\|F(z)\|^2}{\|F(z)+G(z)\|+\|F(z)\|}\ge  \epsilon^{\frac 2 7} ,     \]
which proves (b).

Property (iv) implies that there exists a compact set $L\subset D$ such that 
\begin{equation}
\label{eq:choiceR}
\dist_F(p,bL)>\sigma.
\end{equation}

By enlarging $K$ if necessary, we may assume that $L\Subset \mathring K$.
The part (c) and (d) are proved exactly as in \cite{FF}, and the constant $m$ has to be chosen large enough.
Moreover, $\|G(z)\|$ can be made arbitrarily small for all $z\in K$.
Furthermore, by taking $m$ even larger if necessary, we can assume that $r>0$ is so small that 
\[  L\cap \b(z_{i,j},\lambda r)=\emptyset, \quad \text{for all } 1\le i\le s,\ 1\le j\le N_i.  \]
Since uniform approximation of $F$ on the compact set $K$ implies $\ccal^1$-approximation
of $F$ on the relatively compact subset $L$ 
we get from \eqref{eq:choiceR} that
\begin{equation}\label{oc:22}
\dist_{F+G}(p,bL)>\sigma,
\end{equation}
if  $\|G(z)\|$ is small enough for all $z\in K$.

To prove (e), we consider $\dist_{F+G}(bL,S)$.
Choose any path $\gamma$ in $\cd$ from $S$ to $bL$. Denote its starting point by $q_1\in S$ and 
its ending point by $q_2\in bL$. Choose a ball $\b(z_{i,j},r)$ containing $q_1$. Since $L\cap \b(z_{i,j},\lambda r)=\emptyset$
the path $\gamma$ intersects $b\b(z_{i,j},\lambda r)$; let $q_3$ denote any intersection point.
We have
\begin{equation} 
 \begin{split} \label{oc:1}
   &\length((F+G)(\gamma))\ge\|(F+G)(q_1)-(F+G)(q_3)\|\\
   &\ge \sqrt{|(f_i+g_i)(q_1)-(f_i+g_i)(q_3)|^2+|(f_{i+s}+g_{i+s})(q_1)-(f_{i+s}+g_{i+s})(q_3)|^2}.
  \end{split}
\end{equation}
Note that $|a+b|^2\ge |a|^2-2|b|$ for each $a, b\in\c$ such that $|a|\le 1$.
By Lemma \ref{Lemma5.3FF} and \eqref{eq:6.1FF} we get
\begin{equation*} 
 \begin{split}
  & |(f_i+g_i)(q_1)-(f_i+g_i)(q_3)|^2\\
&=|\beta_{i,j}\phi_{i,j}(q_1)+(g_i(q_1)-\beta_{i,j}\phi_{i,j}(q_1))+(f_i(q_1)-f_i(q_3))
\\
&\quad-\beta_{i,j}\phi_{i,j}(q_3)-(g_i(q_3)-\beta_{i,j}\phi_{i,j}(q_3))|^2\\
&\ge C^2|\beta_{i,j}|^2\eta^{\frac 18}-2(3\eta+\eta^{\frac 23}),
  \end{split}
\end{equation*}
and similar for the index $i+s$ instead of $i$. Therefore by \eqref{oc:1}, \eqref{eq:6.26.3FF} and (iii) we get
\begin{equation} 
 \begin{split} \label{oc:3}
   \length((F+G)(\gamma))\ge& 
   \sqrt{ C^2(|\beta_{i,j}|^2+|\beta_{i+s,j}|^2)\eta^{\frac 18}-4(3\eta+\eta^{\frac 23})}\\
   \ge &\sqrt{C^2(2a\epsilon^{\frac 12}-\epsilon)\eta^{\frac 18}+O(\eta^{\frac 23})}\\
   \ge&\sqrt{2C_1\eta^{\frac 58}+O(\eta^{\frac 23})}\ge\sqrt{C_1}\eta^{\frac5{16}}=E\epsilon^{\frac5{16}},
 \end{split}
\end{equation}
for any $\epsilon\in (0,\epsilon_0)$, provided $\epsilon_0>0$ is small enough; the constants $C_1$ and $E$ depend only on $\epsilon_0$.
Since \eqref{oc:3} holds for any path from $bL$ to $S$ we have $\dist_{F+G}(bL,S)\ge E\epsilon^{\frac5{16}}$, and
by \eqref{oc:22}  we get $\dist_{F+G}(p,S)\ge \sigma + E\epsilon^{\frac5{16}}$, which proves (e).
\end{proof}

\textit{Proof of Theorem \ref{main_theorem}.}
Let $h\colon\cd\to \b_p$ be as in the statement of Theorem \ref{main_theorem}. We shall construct the map
$F$ inductively in a way similar to the proof of \cite[Theorem 1.3]{FF}.

Choose an increasing sequence $\{a_k\}_{k\ge 1}$ converging to 1, and a decreasing sequence $\{\epsilon_k\}_{k\ge 1}$ converging to 0,
such that the following hold:
\begin{enumerate}[\rm(i)]
	\item $\max\{\sup_S\|h\|,\frac12\}<a_1-\epsilon_1^{\frac 12}$,
	\item $\displaystyle{\sum_{k=1}^\infty \epsilon_k^{\frac 12}<\infty},\ \displaystyle{\sum_{k=1}^\infty \epsilon_k^{\frac 5 {16}}=\infty}$,
	\item $a_k+\epsilon_k<a_{k+1}-\epsilon_{k+1}^{\frac 12}$ for all $k\ge 1$.
\end{enumerate}
We can obtain the two sequences as follows:
First we choose $a_1$, $\frac 12<a_1<1$, so close to $1$, and $\epsilon_1>0$ so close to 0 that (i) holds.
Then we choose a decreasing sequence $\{\epsilon_k\}$ converging to 0
such that 
\[
3\sum_{k=1}^\infty \epsilon_k^{\frac 12}=1-a_1,\quad {\sum_{k=1}^\infty \epsilon_k^{\frac 5 {16}}=\infty},
\]
which implies property (ii). The sequence 
\[
a_k=a_1+3\sum_{l=1}^{k-1} \epsilon_l^{\frac 12}+2\epsilon_k^{\frac 12},\quad k\ge 2,
\]
converges to 1 and satisfies (iii).

Let $F_0=(0,\ldots,0,h)$ and fix any $p\in D$. Since $h$ is nonconstant we have $\dist_{F_0}(p,S)>0$.
Let $s$ be the number provided by Lemma \ref{Lemma5.2FF} for  $\lambda$ defined by \eqref{eq:lambda}.
Using Lemma \ref{Lemma6.1FF} we will inductively construct a sequence of entire maps $\{G_j\colon \c^n\to\c^{2s+p}\}_{j\ge 1}$,
a sequence of injective holomorphic immersions  $F_k=F_0+\sum_{j=1}^k G_j$,
two increasing sequences of compact subsets $\{K_k\}_{k\ge 1}$, $\{L_k\}_{k\ge 1}$ of $D$ such that 
\[L_k\Subset \mathring K_k,\quad \text{and }\quad \bigcup_{k=1}^\infty K_k=\bigcup_{k=1}^\infty L_k=D,\]
a decreasing sequence $\{\delta_k\}_{k\ge 1}$ converging to 0, $0<\delta_k<\epsilon_k$, 
such that for every $k\ge 1$ the following properties hold: 
\begin{enumerate}[\rm(a)]
	\item $\|F_{k-1}(z)\|\ge \min_{w\in S}\|F_{k-1}(w)\|-\frac1{2^{k}}$ for each $z\in\cd \setminus K_k$,
	\item $\|F_k(z)\|\le a_k+\epsilon_k$ for each $z\in \cd$,
	\item if $\|F_k(z)\|\le a_{k}-\epsilon_k^{\frac 1 7}$ for some  $z\in S$, then  
	$\|F_k(z)\|>\|F_{k-1}(z)\|+\epsilon_k^{\frac 2 7}$,
	\item $\|G_k(z)\|<\frac{\delta_k}{2^k}$ for each $z\in K_k$,
	\item $\|G_k(z)\|^2<1-\min_{w\in S}\|F_k(w)\|$ for all $z\in \cd$,
  \item $\displaystyle{\dist_{F_{k-1}}(p,bL_{k})>\frac12\dist_{F_{0}}(p,bD)+E\sum_{j=1}^{k-1}\epsilon_j^{\frac 5{16}}}$,
	\item if $F\colon D\to \c^{2s+p}$ is holomorphic and $\|F(z)-F_{k-1}(z)\|<\delta_{k}$ for all $z\in K_{k}$,
	then $\dist_{F}(p,bL_{k})>\dist_{F_{k-1}}(p,bL_{k})-1$.
\end{enumerate}

First choose $L_1$ such that (f) holds for $k=1$, then choose $K_1$, $L_1\Subset \mathring K_1$, such that (a) holds for $k=1$.
Since uniform approximation of $F_0$ on the compact set $K_1$ implies $\ccal^1$-approximation
of $F_0$ on the relatively compact subset $L_1$, there is
$\delta_1$, $0<\delta_1<\epsilon_1$, satisfying (g).
We apply Lemma \ref{Lemma6.1FF} to $F_0$, $a_1$, $\epsilon_1$, $\delta_1/2$ to obtain an entire map
$G_1$, which satisfies properties (b)-(e), and $\dist_{F_1}(p,bD)>\frac12\dist_{F_{0}}(p,bD)+E\epsilon_1^{\frac 5{16}}$.
We proceed similarly, taking (iii) into account and we obtain sequences $G_k$, $K_k$, $L_k$ and $\delta_k$, which satisfy (a)-(g).

Property (d) implies that the sequence $F_k$ converges uniformly on compact sets in $D$ to a holomorphic map $F\colon D\to \c^{2s+p}$
and we get the estimate
\begin{equation} 
\label{eq:proof1}
\begin{split}
\|F_{k-1}(z)-F(z)\|&\le \|F_{k-1}(z)-F_{k}(z)\|+\|F_{k}(z)-F_{k+1}(z)\|+\cdots\\
&\le \frac{\delta_{k}}{2^{k}}+\frac{\delta_{k+1}}{2^{k+1}}+\cdots\le \delta_{k}\quad \text{for every }\quad z\in K_{k}
\end{split}
\end{equation}
This implies together with (f) and (g) that 
\[{\dist_{F}(p,bL_{k})>\frac12\dist_{F_{0}}(p,bD)+E\sum_{j=1}^{k-1}\epsilon_j^{\frac 5{16}}}-1.\]
By (ii) the series 
$\sum_{j}\epsilon_j^{\frac 5{16}}$ 
 diverges, which implies that the map $F$ is  complete.
Property (b) and the maximum principle imply that $F(D)\subset \b_{2s+p}$.
Since the map $h$ is an injective immersion on $D$, $F_0=(0,\ldots,0,h)$ and all last $p$ components of the maps $G_k$ are zero
for each $k$, all the maps $F_k$ and the limit map $F$ are injective immersions.
The fact that the map $F$ is proper is proved as in \cite{FF},
where we take into account that the series
$\sum_j\epsilon_j^{\frac27}$ is divergent by (ii).
This completes the proof of Theorem \ref{main_theorem}.\hfill \qed

\subsection*{Acknowledgements}

The author is partially  supported  by the research program P1-0291 and the grant J1-5432 from ARRS, Republic of Slovenia.

Part of this work was done  while the author was visiting University of Oslo. She would like 
to thank the Department of Mathematics for the hospitality and partial financial support.


\vskip 0.4cm

\noindent Faculty of Mathematics and Physics, University of Ljubljana, and Institute
of Mathematics, Physics and Mechanics, Jadranska 19, SI--1000 Ljubljana, Slovenia.

\noindent e-mail: {\tt barbara.drinovec@fmf.uni-lj.si}


\begin{thebibliography}{12}

\bibitem{AF}
Alarc\'on, A.;  Forstneri\v c, F.: 
Every bordered Riemann surface is a complete proper curve in a ball. 
Math.\ Ann. {\bf 357} (2013) 1049--1070

\bibitem{AF1}
Alarc\'on, A.;  Forstneri\v c, F.: 
Null holomorphic curves in $\c^3$ and applications to the
conformal Calabi-Yau problem.
\texttt{arXiv:1311.1985}

\bibitem{Ca}
Calabi, E.: Problems in differential geometry. Ed. S. Kobayashi and J. Eells, Jr., Proceedings of the United
States-Japan Seminar in Differential Geometry, Kyoto, Japan, 1965. Nippon Hyoronsha Co., Ltd., Tokyo
(1966) 170

\bibitem{Fo}
Fornaess, J.E.:
Embedding Strictly Pseudoconvex Domains in Convex Domains.
Amer.\ J.\ Math. {\bf 98} (1976) 529--569  

\bibitem{FF}
Forstneri\v c, F.:
{Embedding strictly pseudoconvex domains into balls}.
Trans.\ Amer.\ Math.\ Soc. {\bf295} (1986) 347--368

\bibitem{FW}
Forstneri\v c, F.; Wold, E.F.: 
Bordered Riemann surfaces in $\c^2$. 
J.\ Math.\ Pures Appl. (9) \textbf{91} (2009) 100--114 


\bibitem{Gl}
Globevnik, J.:
A complete complex hypersurface in the ball of $\c^N$.
\texttt{arXiv:1401.3135}

\bibitem{Gl2}
Globevnik, J.:
Boundary continuity of complete proper holomorphic maps. 
J.\ Math.\ Anal.\ Appl.\ \textbf{424} (2015) 824--825 

\bibitem{HS}
Hakim, M.;  Sibony N.:
{Fonctions holomorphes born\'ees sur la boule unit\'e de $\c^n$}.
Invent.\ Math. {\bf{67}} (1982) 213--222

\bibitem{Jo}
Jones, P. W.: 
A complete bounded complex submanifold of $\c^3$.
Proc.\ Amer.\ Math.\ Soc. {\bf {76}} (1979) 305--306

\bibitem{JX}
Jorge, L.P.; Xavier, F.: 
A complete minimal surface in $\r^3$ between two parallel planes.
Ann.\ of Math. {\bf{112}} (1980) 203–-206


\bibitem{Lo1}
L\o w E.:
{A construction of inner functions on the unit ball in $\c^p$}.
Invent.\ Math. {\bf{67}} (1982) 223--229


\bibitem{Lo2}
L\o w E.:
{Inner functions and boundary values in $H^\infty(\Omega)$ and $A (\Omega)$ in smoothly bounded
pseudoconvex domains}.
Math.\ Z. {\bf{185}} (1984) 191--210  


\bibitem{Lo3}
L\o w E.:
{Embeddings and proper holomorphic maps of strictly pseudoconvex domains into
polydiscs and balls,}.
Math.\ Z. {\bf{190}} (1985) 401--410

\bibitem{Lo4}
L\o w E.:
The ball in $\c^n$ is a closed complex submanifold of a polydisc.
Invent.\ Math. {\bf{83}} (1986) 405--410


\bibitem{Yo1}
Yang, P.:
{Curvatures of complex submanifolds of {${\c}^{n}$}}.
J.\ Differential Geom. {\bf 12} (1978) 499--511 

\bibitem{Yo2}
Yang, P.:
Curvature of complex submanifolds of {$\c^{n}$}.
Several complex variables ({P}roc.\ {S}ympos.\ {P}ure {M}ath.,
              {V}ol. {XXX}, {P}art 2, {W}illiams {C}oll., {W}illiamstown,
              {M}ass., 1975), {135--137}, {Amer. Math. Soc., Providence, R.I.} (1977)

\end{thebibliography}
\end{document}